\documentclass[11pt]{amsart}
\usepackage{geometry,graphicx}

\usepackage{latexsym}

\usepackage{multirow}
\usepackage{ulem}
\usepackage{xcolor,colortbl}
\usepackage[absolute,overlay]{textpos}
\usepackage{graphicx}
\usepackage{xcolor,colortbl}
\usepackage{ctable}

\usepackage{amsmath, amssymb, mathtools}

\usepackage{import}
\usepackage{xifthen}
\usepackage{pdfpages}
\usepackage{transparent}
\usepackage{comment}

\usepackage{amssymb}
\usepackage[cp850]{inputenc}
\usepackage{epsfig}
\usepackage{epstopdf}
\usepackage{psfrag}
\usepackage{amsthm}
\usepackage{amscd}
\usepackage{amsmath}
\usepackage{amsfonts}
\usepackage{graphics}
\usepackage{enumerate}
\usepackage{color}
\theoremstyle{theorem}

\newcommand{\set}[1]{\left\{ #1 \right\}}
\newcommand{\abs}[1]{| #1 |}
\newcommand{\pes}[1]{\langle{#1}\rangle}
\newcommand{\norm}[1]{\left\lVert#1\right\rVert}
\newcommand{\stnorm}[1]{\left\lVert#1\right\rVert_{\text{st}}}

\newcommand{\Z}{\mathbb{Z}}
\newcommand{\R}{\mathbb{R}}

\newcommand{\T}{\mathbb{T}}

\DeclareMathOperator{\sys}{sys}
\DeclareMathOperator{\vol}{vol}

\DeclareMathOperator{\area}{area}

\DeclareMathOperator{\conv}{conv}

\newtheorem{theorem}{Theorem}

\newtheorem{proposition}{Proposition}

\newtheorem{lemma}{Lemma}

\theoremstyle{definition}

\newtheorem{definition}{Definition}

\newtheorem*{remark}{Remark}

\numberwithin{equation}{section}

\numberwithin{theorem}{section}

\numberwithin{definition}{section}

\numberwithin{lemma}{section}

\numberwithin{conjecture}{section}

\numberwithin{proposition}{section}

\usepackage{import}

\usepackage{xifthen}

\usepackage{pdfpages}

\usepackage{transparent}

\definecolor{primarycream2}{RGB}{230,230,230}

\definecolor{secondarycream2}{RGB}{240,240,240}

\begin{document}

%




%

\title[In the Finsler shadow of Loewner]{Isosystolic inequalities on two-dimensional Finsler tori}




\author[F.~Balacheff]{Florent Balacheff}

\address{Florent Balacheff, Departament de Matem\`atiques, Universitat Aut\`onoma de Barcelona, Barcelona, Spain}

\email{Florent.Balacheff@uab.cat}

\author[T.~Gil Moreno de Mora]{ Teo Gil Moreno de Mora}

\address{Teo Gil Moreno de Mora, Departament de Matem\`atiques, Universitat Aut\`onoma de Barcelona, Barcelona, Spain \& LAMA, Universit\'e Paris-Est Cr\'eteil, Cr\'eteil, France}

\email{teo.gil-moreno-de-mora-i-sarda@u-pec.fr}

\date{\today}

\keywords{Busemann-Hausdorff and Holmes-Thompson area, isosystolic inequalities, Finsler metrics, stable norm, systole.}
\subjclass[2010]{Primary: 53C23}
\thanks{The first author acknowledges support by the FSE/AEI/MICINN grant RYC-2016-19334  ``Local and global systolic geometry and topology", the FEDER/AEI/MICIU grant PGC2018-095998-B-I00 ``Local and global invariants in geometry" and the AGAUR grant 2017-SGR-1725. The second author acknowledges support by the Spanish MEFP grant BDNS 512590 ``Beca de Colaboraci\' on", the ``Fondation Sciences Math\'ematiques de Paris" PGSM Master scholarship and the ``\'Ecole Doctorale MSTIC" doctoral fellowship.}




%

\begin{abstract} 
In this article we survey all known optimal isosystolic inequalities on two-dimensional Finsler tori involving the following two central notions of Finsler area: the Busemann-Hausdorff area and the Holmes-Thompson area. We also complete the panorama by establishing the following new optimal isosystolic inequality that is deduced from prior work by Burago and Ivanov: the Busemann-Hausdorff area of a Finsler reversible $2$-torus with unit systole is at least equal to $\pi/4$.
\end{abstract}

\maketitle

\section{Introduction}
The purpose of this article is to present a self-contained survey on several optimal isosystolic inequalities for the two-dimensional torus, and to etablish a new one. We will consider length metrics arising from infinitesimal convex structures, namely Finsler (reversible or not) metrics. It includes smooth Riemannian metrics as a special case. For such length metrics, there exist various notions of area, and we choose to focus on the following two central notions: the Busemann-Hausdorff area and the Holmes-Thompson area. These two notions are particularly relevant for isosystolic inequalities: Busemann-Hausdorff area generalizes the notion of Hausdorff measure, and is natural from the metric point of view, while Holmes-Thompson area is a symplectic invariant of the geodesic flow, and is natural from the dynamic point of view. 

\subsection{Finsler metrics}\label{subsec:Finsler}
A {\it (continuous) Finsler metric} on the $2$-torus $\T^2$ is a continuous function $F:T\T^2\to \R_+$ such that the restriction $F(x,\cdot)$ to each tangent space $T_x\T^2$ is a norm that we denote by $\|\cdot\|^F_x$. Let us emphasize that we do not require the norm to be symmetric, but only to be a function positive outside the origin, convex and positively homogeneous. In particular, the subset of vectors $v$ in $T_x\T^2$ satisfying $\|v\|^F_x\leq 1$ is a convex body $K_x\subset T_x\T^2$ containing the origin in its interior. Therefore, a Finsler metric amounts to a collection $\{K_x\}_{x \in \T^2}$ of convex bodies that continuously depends on the point $x$. If each one of these convex bodies is symmetric, the metric is said to be {\it reversible}. If each one is an ellipse centered at the origin that smoothly depends on the point, the metric is said to be {\it Riemannian}, and in particular, is reversible.
Denote by $\pi : \R^2 \to \T^2$ the universal covering map obtained by identifying $\T^2$ with  the quotient space $\R^2/\Z^2$. A Finsler metric $F$ on $\T^2$ induces a $\Z^2$-periodic Finsler metric $\tilde{F}$ on $\R^2$ through the formula $\tilde{F}(\tilde{x},\tilde{v})=F(\pi(\tilde{x}),d_{\tilde{x}}\pi(\tilde{v}))$. Using the canonical identification $T_{\tilde{x}}\R^2\simeq \R^2$, the associated collection of convex bodies $\widetilde{K}_{\tilde{x}}\subset T_{\tilde{x}}\R^2$ thus defines a continuous $\Z^2$-periodic map 
\begin{eqnarray*}
\widetilde{K}: \R^2 &\to &\mathcal{K}_0(\R^2)\\
\tilde{x}&\mapsto&\widetilde{K}_{\tilde{x}}=d_{\tilde{x}}\pi^{-1}(K_{\pi(\tilde{x})})
\end{eqnarray*} 
where  $\mathcal{K}_0(\R^2)$ denotes the space of convex bodies in $\R^2$ containing the origin in their interior and endowed with the Hausdorff topology. 
If the above map is constant, we will say that the Finsler metric on $\T^2$ is {\it flat}. And for flat metrics, we will denote both $\widetilde{K}_{\tilde{x}}$ and $K_x$ simply by $K$. Remark that Riemannian flat metrics on $\T^2$ can be classified using quotients of the Euclidean plane by full rank lattices, fact that will be useful in section \ref{sec:riem}. The situation is completely different in the Finsler case, as the space of flat Finsler metrics is already huge. The use of lattices in this context is not particularly decisive, and this is why we have decided to fix once and for all the lattice to be $\Z^2$ when describing $\T^2$ as a quotient of $\R^2$.

\subsection{Systole and Finsler areas}
Given a Finsler metric $F$ on  $\T^2$, the length of a piecewise smooth curve $\gamma : [a,b] \to \T^2$ is defined using the formula
$$
\ell_F(\gamma):=\int_a^b \|\dot{\gamma}(t)\|^F_{\gamma(t)}dt.
$$
This length functional gives rise to a Finsler distance $d_F$ on $\T^2$ obtained by minimizing the length of such curves connecting two given points. This Finsler distance may be not symmetric if the metric is not reversible. 
 A geodesic is a curve which is everywhere locally a distance minimizer.

We now present the first ingredient for isosystolic inequalities, namely the systole.
\begin{definition}
Given a Finsler metric $F$ on $\T^2$, the {\it systole} is defined as the quantity
$$
\sys(\T^2,F):=\inf \{\ell_F(\gamma) \mid \gamma \, \, \text{non-contractible closed curve in} \, \, \T^2\}.
$$
\end{definition}
It is easy to see that the systole can be read on the universal cover of the two-torus using the formula
$
\sys(\T^2,F)=\min \{d_{\tilde{F}}(\tilde{x},\tilde{x}+z) \mid \tilde{x} \in [0,1]^2 \, \, \text{and} \, \, z \in \Z^2\setminus\{0\}\}.
$
In particular, the value $\sys(\T^2,F)$ is always positive and the infimum is actually a minimum realized by the length of a shortest non-contractible closed geodesic. 

The second ingredient for isosystolic inequalities is the two-dimensional volume, or area. For Finsler manifolds, there exist many notions of volume, but in this article we will be interested in the following two central notions.
First recall that given a convex body $\widetilde{K}\subset \R^2$ containing the origin in its interior, its polar body is the convex body defined by
$\widetilde{K}^\circ:=\{\tilde{x} \in \R^2 \mid \langle \tilde{x},\tilde{y}\rangle\leq 1 \, \, \text{for all} \, \, \tilde{y} \in \widetilde{K}\}$
where $\langle \cdot,\cdot \rangle$ denotes the Euclidean scalar product of $\R^2$.
Denote by $|\cdot|$ the standard Lebesgue measure on $\R^2$.

\begin{definition}\label{def:area}
The {\it Busemann-Hausdorff area} of a Finsler $2$-torus $(\T^2,F)$ is defined as the quantity
$$
\area_{BH}(\T^2,F):=\int_{[0,1]^2} {\pi\over |\widetilde{K}_x|}\, d\tilde{x}_1 d\tilde{x}_2,
$$
and its {\it Holmes-Thompson area} is defined as
$$
\area_{HT}(\T^2,F):=\int_{[0,1]^2} {|\widetilde{K}^\circ_x|\over \pi}\, d\tilde{x}_1  d\tilde{x}_2.
$$
\end{definition}

So the Busemann-Hausdorff notion of area corresponds to integrating over a fundamental domain the unique multiple of the Lebesgue measure for which the measure of $\widetilde{K}_x$ equals the Lebesgue measure of the Euclidean unit disk, while the Holmes-Thompson notion of area corresponds to integrating the unique multiple of the Lebesgue measure for which the measure of the polar body $\widetilde{K}^\circ_x$ equals the Lebesgue measure of the Euclidean unit disk.
When the convex $K$ is symmetric, Blaschke's inequality \cite{Blaschke} asserts that $|K|\cdot|K^\circ|\leq \pi^2$ with equality if and only if $K$ is an ellipse. Therefore the inequality $\area_{HT}(\T^2,F)\leq \area_{BH}(\T^2,F)$ holds true for Finsler reversible metrics, with equality if and only if $F$ is a continuous Riemannian metric. Further observe that for Riemannian metrics both notions of area coincide with the standard notion of Riemannian area.

It is worth saying that these two notions of area do not depend on the specific choice of the Euclidean scalar product in Definition \ref{def:area}. In fact, both notions admit an intrinsic definition we will not present here. Let us just mention that Busemann-Hausdorff area coincides with the $2$-dimensional Hausdorff measure for reversible metrics (see \cite{Busemann 1947}), while Holmes-Thompson area coincides with the symplectic volume (normalized by a suitable constant) of the bundle of the dual convex bodies in $T^\ast \T^2$ (see \cite{HT79}).

\subsection{Optimal isosystolic inequalities}
Given a choice  denoted by $\area_\ast$ with $\ast=BH$ or $HT$ of one of these two notions of area, the {\it systolic $\ast$-area} of a Finsler metric $F$ is defined as the quotient
$$
\frac{\area_\ast(\T^2,F)}{\sys(\T^2,F)^2}.
$$ 
Observe that this functional is invariant by rescaling the metric $F$ into $\lambda F$ for any positive constant $\lambda$.
An {\it isosystolic inequality} is then a positive lower bound on the systolic $\ast$-area holding for a large class of metrics. Equivalently, it amounts to an inequality of the type 
$$
\area_\ast(\T^2,F)\geq C \cdot \sys(\T^2,F)^2
$$
for some positive constant $C$. If the constant $C$ can not be improved, the corresponding isosystolic inequality is said to be {\it optimal}. In the absence of isosystolic inequality, that is when the infimum of the systolic $\ast$-area function over some class of metrics is zero, we say that {\it systolic freedom} holds.
 
 In this paper we will be concerned with the following classes of metrics: flat Riemannian metrics, Riemannian metrics, flat Finsler reversible metrics, flat Finsler metrics, Finsler reversible metrics, and finally Finsler metrics.
 Here is a table summarizing the currently known optimal  isosystolic inequalities on $\T^2$ for these classes of metrics and the two notions of area we are interested in.
 
\bigskip

\begin{table}[h!]\label{table}

\caption{Optimal constants $C$ for several classes of Finsler metrics}

\begin{tabular}{c@{\hspace{0.45cm}}c@{\hspace{0.45cm}}c}
    \rowcolor{primarycream2} \cellcolor{white} & Reversible & Non-reversible  \\ \arrayrulecolor{white}\specialrule{0.3mm}{0.3mm}{0.3mm}
    
    \rowcolor{secondarycream2} \cellcolor{primarycream2} 
    \begin{tabular}{@{}c@{}} \cellcolor{primarycream2} Flat Riemannian \\ \cellcolor{primarycream2} metrics\end{tabular} & \begin{tabular}{@{}c@{}}Folklore \\ $ \sqrt{3}/2$ \end{tabular} & $\times$ \\
    \arrayrulecolor{white}\specialrule{0.3mm}{0.3mm}{0.3mm}
        
    \rowcolor{secondarycream2} \cellcolor{primarycream2} \begin{tabular}{@{}c@{}} \cellcolor{primarycream2} Riemannian \\ \cellcolor{primarycream2} metrics \end{tabular} & \begin{tabular}{@{}c@{}}Loewner 1949 \\ $\sqrt{3}/2$ \end{tabular} & $\times$ \\
    \arrayrulecolor{white}\specialrule{0.3mm}{0.3mm}{0.3mm}
        
    \rowcolor{secondarycream2} \cellcolor{primarycream2} \begin{tabular}{@{}c@{}} \cellcolor{primarycream2} Flat Finsler metrics \\ \cellcolor{primarycream2} BH-area \end{tabular} & \begin{tabular}{@{}c@{}}Minkowski 1896 \\ $\pi/4$ \end{tabular} & \begin{tabular}{@{}c@{}} Systolic freedom \\ $0$ \end{tabular} \\
    \arrayrulecolor{white}\specialrule{0.3mm}{0.3mm}{0.3mm}
        
    \rowcolor{secondarycream2} \cellcolor{primarycream2} \begin{tabular}{@{}c@{}} \cellcolor{primarycream2} Finsler metrics \\ \cellcolor{primarycream2} BH-area \end{tabular} & \begin{tabular}{@{}c@{}}Open \\  ? \end{tabular} & \begin{tabular}{@{}c@{}} Systolic freedom \\ $0$ \end{tabular} \\
    \arrayrulecolor{white}\specialrule{0.3mm}{0.3mm}{0.3mm}
        
    \rowcolor{secondarycream2} \cellcolor{primarycream2} \begin{tabular}{@{}c@{}} \cellcolor{primarycream2} Flat Finsler metrics \\ \cellcolor{primarycream2} HT-area \end{tabular} & \begin{tabular}{@{}c@{}}Minkowski + Mahler \\ $2/\pi$ \end{tabular} & \begin{tabular}{@{}c@{}} \' Alvarez-B-Tzanev 2016 \\ $3/2\pi$ \end{tabular} \\
    \arrayrulecolor{white}\specialrule{0.3mm}{0.3mm}{0.3mm}
        
    \rowcolor{secondarycream2} \cellcolor{primarycream2} \begin{tabular}{@{}c@{}} \cellcolor{primarycream2} Finsler metrics \\ \cellcolor{primarycream2} HT-area \end{tabular} & \begin{tabular}{@{}c@{}} Sabourau 2010 \\ $2/\pi$ \end{tabular} & \begin{tabular}{@{}c@{}} \' Alvarez-B-Tzanev 2016 \\ $3/2\pi$ \end{tabular}
\end{tabular}
\bigskip
\end{table}
 
  It is important to observe that all the flat isosystolic optimal inequalities here are implied by the corresponding non-flat ones. This is not a coincidence as the proof for all the non-flat classes in the above table always proceeds with the same strategy: first find a flat metric in this class with lower systolic area, and then apply the corresponding flat isosystolic inequality. However we have decided to express the flat case of these isosystolic inequalities in independent statements to underline  their own importance, and their connection with several fundamental results in the geometry of numbers and convex geometry as we shall later explain.

 Let us emphasize that until now the question of finding the optimal lower bound for the systolic BH-area in the reversible Finsler case remained open, as reflected in the above table. From the fact that $\area_{BH}\geq \area_{HT}$ and Sabourau's isosystolic inequality, we easily obtain  that for a Finsler reversible metric $F$ on $\T^2$ the following inequality holds true: $\area_{BH}(\T^2,F)\geq \frac{2}{\pi} \sys^2(\T^2,F)$. 
 Never\-theless it is reasonable to think that the optimal constant should be $\pi/4$ like in the flat case, as suggested to the first author by J.C. \'Alvarez Paiva in a private conversation. In section \ref{sec:BH} we will establish this conjecture: 

\begin{theorem}\label{th:opt}
Let $F$ be a Finsler reversible metric on $\T^2$. Then the following optimal inequality holds true:
\[
\area_{BH}(\T^2,F)\geq \frac{\pi}{4} \sys^2(\T^2,F).
\]
 Equality holds for the flat metric corresponding to the supremum norm $\|\cdot\|_\infty$.
\end{theorem}

 This result finally settles the search for optimal isosystolic inequalities on two-dimensional Finsler tori for the Busemann-Hausdorff area and the Holmes-Thompson area, and definitely complete Table 1.
 
The main step in the proof of Theorem \ref{th:opt} is related to the asymptotic geometry of  universal covers of Finsler tori. Namely, according to \cite{burago} the universal cover of a Finsler torus admits a unique norm---called the {\it stable norm}---which asymptotically approximates the pullback Finsler metric, see formula (\ref{eq:st}). This norm gives rise to a flat Finsler metric on $\T^2$ abusively also called stable norm.
Theorem \ref{th:opt}  is then a consequence (see section \ref{sec:BH}) of the following statement:
\smallskip

\noindent {\it ($\ast)$ passing from a Finsler reversible metric on a $2$-torus to its stable norm decreases the Busemann-Hausdorff area}.
\smallskip

\noindent It turns out that this result can be deduced from prior work by Burago and Ivanov. More precisely, in \cite{burago_ivanov_2012}, they proved that
 \smallskip

\noindent  {\it ($\ast \ast $) a region in a two-dimensional affine subspace of a normed space has the least Hausdorff area among all compact surfaces with the same boundary},
 \smallskip

\noindent  which has been shown to be equivalent to statement ($\ast$) in \cite[Theorem 1]{burago_ivanov}. 
In the present paper, our main contribution is to propose a direct proof of statement ($\ast$) based on the ideas of Burago and Ivanov, see section \ref{sec:decreasing_area_BH}.

\subsection{Organization of the paper.}

 In section \ref{sec:riem} we start by explaining how to derive the optimal systolic area for flat Riemannian metrics and its connection to the two-dimensional Hermite constant, and next prove Loewner's optimal isosystolic inequality for Riemannian metrics. In section \ref{sec:flatFinsler} we derive all the optimal isosystolic inequalities for flat Finsler metrics. More precisely, we first treat the case of Busemann-Hausdorff area in the reversible case, which reduces to Minkowski's first theorem, and explain why systolic freedom appears in the non-reversible case. In a second time we focus on Holmes-Thompson area, treating first the reversible case which is deduced from a combination of Minkowski's first theorem and Mahler's volume product inequality in dimension $2$, and secondly the non-reversible case by sketching the arguments appearing in \cite{ABT}. In section \ref{sec:stable_norm} we present the stable norm, which describes the asymptotic geometry of the universal covering space of Finsler tori, and the associated notion of calibrating functions introduced in \cite{burago_ivanov} together with their main properties that will be needed in section \ref{sec:HT} and \ref{sec:BH} in order to prove all the optimal inequalities for (non-flat) Finsler metrics. In section \ref{sec:HT}, we prove two optimal isosystolic inequalities on the $2$-torus for the Holmes-Thompson notion area: one for reversible Finsler metrics, and another for (possibly non-reversible) Finsler ones. 
  The same strategy applies for both proofs and proceeds as follows: first prove that the associated stable norm has smaller systolic area than the original metric, and then apply the corresponding flat optimal systolic inequality already proved in section \ref{sec:flatFinsler}. The main step in the proof boils down to the following result due to \cite{burago_ivanov} : passing for a $2$-torus from a Finsler metric to its stable norm decreases the Holmes-Thompson area. Finally in section \ref{sec:BH}, we prove the new optimal isosystolic inequality on reversible Finsler $2$-tori for Busemann-Hausdorff area. Here again the main step boils down to prove that passing for a $2$-torus from a reversible Finsler metric to its stable norm decreases the Busemann-Hausdorff area. To conclude this last section, we describe a counterexample to this Busemann-Hausdorff area decreasing property in the non-reversible case, communicated to us by Ivanov and reproduced here with his kind permission.

\section{Isosystolic inequalities for flat and non-flat Riemannian metrics}\label{sec:riem}

We first explain how the optimal isosystolic inequality for flat Riemannian metrics on $\T^2$ reduces to compute the Hermite constant in dimension $2$. Then we prove Loewner's theorem that states the optimal isosystolic inequality for Riemannian metrics on $\T^2$. Recall that for Riemannian metrics, both Busemann-Hausdorff  and Holmes-Thompson notions of area coincide with the standard Riemannian area simply denoted by $area$ in this section.

\subsection{The flat Riemannian case: Hermite constant in dimension $2$.}
We start by recalling the definition of Hermite constant in an arbitrary dimension $n$.
Given a full rank lattice $L$ in $\R^n$ endowed with the standard Euclidean structure $\langle \cdot,\cdot \rangle$, its determinant $\det(L)$ is defined as the absolue value of the determinant of any of its basis, while its norm $N(L)$ is defined as the minimum value $\langle \lambda,\lambda \rangle$ over all elements $\lambda \in L\setminus \{0\}$. The Hermite invariant of $L$ is then defined as the quantity
$$
\mu(L)=\frac{N(L)}{\det(L)^{2\over n}}
$$
and the Hermite constant $\gamma_n$  as the supremum value of Hermite invariant $\mu(L)$  over all full rank lattices $L$ of $\R^n$. The Hermite constant is finite in every dimension, and its exact value is known only for dimensions $n=1,\ldots,8$ and $24$. Asymptotically it behaves like ${n\over 2\pi e}$ up to a factor 2, see \cite{CS98}. We are interested in the following statement.

\begin{theorem}[Folklore]
$\gamma_2={2 \over \sqrt{3}}$.
\end{theorem}

\begin{proof}
Because the Hermite invariant $\mu(\cdot)$ is invariant under scaling  and rotating the lattice, we can suppose that $L=\Z(1,0)\oplus \Z v$ where $v=(v_1,v_2)$ has norm at least $1$.  By possibly changing $v=(v_1,v_2)$ into $v=(v_1,-v_2)$ in the preceding expression of $L$ (which corresponds to a reflection  of the lattice along the $x$ axis that does not change the value of $\mu(L)$), we can also suppose that $v_2>0$. Finally by possibly replacing $v$ by $v+n(1,0)$ for some $n \in \Z$ (which does not change $L$ at all), we can suppose that $v$ belongs to the domain $|v_1|\leq 1/2$, $v_2>0$ and $v_1^2+v_2^2\geq 1$. We have $N(L)=1$ as a shortest vector is $u=(1,0)$. It is then straightforward to check that $\det(L)=\det(u,v)=v_2$ is minimal when the second coordinate of $v$ is minimal, that is for $v=(\pm1/2,\sqrt{3}/2)$.
\end{proof}

Note that this supremum is reached if and only if $L$ is an hexagonal lattice (that is, a lattice generated by two vectors forming an angle of $2\pi/3$ and of equal lengths). 

Now observe that any flat two-dimensional Riemannian torus is isometric to the quotient of the Euclidean plane by some lattice $L$ (just choose a linear map $T:\R^2 \to \R^2$ whose sends the ellipse formed by unit vectors at some (and so any) point to a circle and set $L=T(\Z^2)$). For such a flat $2$-torus $\T^2_L:=(\R^2/L,\langle\cdot,\cdot\rangle)$ we easily find that $\sys(\T^2_L)=\sqrt{N(L)}$ while $\area(\T^2_L)=\det(L)$. So the previous result amounts to the following optimal isosystolic inequality for flat Riemannian $2$-tori.

\begin{theorem}[Hermite constant in dimension $2$, systolic formulation]\label{th:flatRIemannian}
Let $g$ be a Riemannian flat metric on $\T^2$. Then the following holds true:
\[
\area(\T^2,g)\geq \frac{\sqrt{3}}{2}\sys^2(\T^2,g).
\]
Furthermore equality holds if and only if $(\T^2,g)$ is isometric to the quotient of the Euclidean plane by some hexagonal lattice.
\end{theorem}

\subsection{The Riemannian case: Loewner's isosystolic inequality}
We now prove Loewner's theorem (unpublished, see \cite{Pu52}), which describes the optimal isosystolic inequality for two-dimensional Riemannian tori.

\begin{theorem}[Loewner's isosystolic inequality, 1949]
Let $g$ be a Riemannian metric on $\T^2$. Then the following holds true:
\begin{equation}\label{eq:Loewner}
\area(\T^2,g)\geq \frac{\sqrt{3}}{2}\sys^2(\T^2,g).
\end{equation}
Furthermore equality holds if and only if  $(\T^2,g)$ is isometric to the quotient of the Euclidean plane by some hexagonal lattice.
\end{theorem}

\begin{proof} 
First recall that Riemannian metrics are assumed to be smooth, see subsection \ref{subsec:Finsler}.
The uniformization theorem (see \cite{MT02} for instance) ensures that $g$ is isometric to a metric of the form $f g_0$ where $f$ is a positive smooth function on $\T^2$ and $g_0$ a flat metric obtained as the quotient of the plane with its Euclidean structure by some full rank lattice. We observe in particular that $g_0$ always admits a transitive compact subgroup ${\mathcal I}$ of isometries corresponding to Euclidean translations. This compact Lie group possesses a unique normalized  Haar measure $\mu$. 
Denote by
\[
\bar{f}:=\int_{\mathcal I} (f \circ I) d\mu
\]
the averaged conformal factor to which we associate the new Riemannian metric $\bar{g}=\bar{f} g_0$.
First observe that
\begin{eqnarray*}
\area(\T^2,\bar{g})&=&\int_{\T^2}\, dv_{\bar{g}}\\
&=&\int_{\T^2}\int_{\mathcal I} (f \circ I) \, d\mu \, dv_{g_0}\\
&=&\int_{\mathcal I}\int_{\T^2} (f \circ I) \, dv_{g_0} \, d\mu \qquad  \text{(by Fubini)} \\
&=&\int_{\mathcal I}\int_{\T^2} f \, dv_{g_0} \, d\mu  \qquad  \text{($I$ being an isometry of $g_0$)}\\
&=&\int_{\mathcal I} \area(\T^2,g) \, d\mu\\
&=& \area(\T^2,g).
\end{eqnarray*}
Besides, for any non-contractible closed curve $\gamma:S^1\to \T^2$, we have
\begin{eqnarray*}
\ell_{\bar{g}}(\gamma)&=&\int_{S^1} \|\dot{\gamma}(t)\|^{\bar{g}}_{\gamma(t)} dt=\int_{S^1} \sqrt{\bar{f}}\cdot\|\dot{\gamma}(t)\|^{g_0}_{\gamma(t)} dt\\
&=&\int_{S^1} \sqrt{\int_{\mathcal I} (f \circ I) \, d\mu}  \cdot \|\dot{\gamma}(t)\|^{g_0}_{\gamma(t)} dt\\
&\geq&\int_{S^1} \int_{\mathcal I}  \sqrt{f \circ I} \, d\mu  \cdot \|\dot{\gamma}(t)\|^{g_0}_{\gamma(t)} dt \qquad   \text{(by Jensen's inequality)}\\
&=&\int_{\mathcal I} \int_{S^1}  \sqrt{f \circ I}\cdot \|\dot{\gamma}(t)\|^{g_0}_{\gamma(t)} dt \, d\mu  \qquad  \text{(by Fubini again)} \\
&=&\int_{\mathcal I} \ell_g(I\circ \gamma) \, d\mu  \\
&\geq&\int_{\mathcal I} \sys(\T^2,g) \, d\mu \qquad  \text{($I\circ \gamma$ being non-contractible)}\\
&= &\sys(\T^2,g).
\end{eqnarray*}
Therefore we conclude that $\sys(\T^2,\bar{g})\geq \sys(\T^2,g)$, and consequently that
\begin{eqnarray}\label{eq:average}
\frac{\area(\T^2,g)}{\sys^2(\T^2,g)}\geq \frac{\area(\T^2,\bar{g})}{\sys^2(\T^2,\bar{g})}.
\end{eqnarray}

Now using the transitivity of $\mathcal I$, we observe that the function $\bar{f}$ is constant as for any $I \in {\mathcal I}$ we have $\bar{f}\circ I=\bar{f}$ by construction. By homogeneity of the systolic area, we finally deduce that 
$$
\frac{\area(\T^2,g)}{\sys^2(\T^2,g)}\geq \frac{\area(\T^2,g_0)}{\sys^2(\T^2,g_0)}.
$$
Therefore we derive from Theorem \ref{th:flatRIemannian} the desired inequality.

The equality case in (\ref{eq:average}) occurs if and only if $f \circ I$ is constant for all $I \in {\mathcal I}$, that is when $f$ itself is constant and the metric $g$ is isometric to a flat one. Therefore the equality case in (\ref{eq:Loewner}) occurs if and only if $(\T^2,g)$ is isometric to the quotient of the Euclidean plane by some hexagonal lattice.
\end{proof}

\section{Isosystolic inequalities for flat Finsler metrics}\label{sec:flatFinsler}

In this section, we survey optimal isosystolic inequalities for Finsler flat metrics on the two-torus for the two notions of area we are interested in.

\subsection{Busemann-Hausdorff area in the flat reversible case}

First recall the celebrated foundational result of the geometry of numbers in the $2$-dimensional case, see \cite{Minkowski}.

\begin{theorem}[Minkowski's first theorem, 1896]\label{th:firstMin}
Let $K \subset \R^2$ be a symmetric convex body such that $\text{int}(K) \cap \Z^2 = \set{0}$. Then its Lebesgue measure satisfies 
$$
\abs{K} \leq 4.
$$
Equality holds when $K$ is the unit disc of the supremum norm $\|\cdot\|_\infty$.
\end{theorem}

In fact, it is even true that  equality holds if and only if  $K$ is the image of the previous square under some element of $SL_2(\Z)$. But we will not need this fact.
\begin{proof}
We argue by contradiction as follows.
Consider $\T^2=\R^2/\Z^2$ endowed with the Riemannian metric induced by the Euclidean scalar product.  If $\abs{K} > 4$, fix $0<\lambda<1$ such that the symmetric convex body $K'=\lambda \cdot K\subset \text{int}(K)$ still satisfies $\abs{K'} > 4$. Then the homothetic symmetric convex body $K'/2$ would have Lebesgue measure strictly greater than $1$. Thus the universal covering map $\pi : \R^2 \to \T^2$ restricted to $K'/2$ cannot be injective, as in the contrary it would imply that $\abs{K'/2}=\vol(\pi(K'/2))\leq\vol(\T^2)=1$. Therefore there exist two points $\tilde{x}$ and $\tilde{x}+z$ with $z \in \Z^2\setminus\{0\}$ both belonging to $K'/2$. As $K'/2$ is symmetric, we get that $-\tilde{x}$ also belongs to $K'/2$, which ensures by convexity that 
$\frac{z}{2}=\frac{\tilde{x}+z-\tilde{x}}{2} \in K'/2 \Leftrightarrow z \in K'\subset\text{int}(K)$: a contradiction.
\end{proof}

Observe that the same proof actually works in arbitrary dimension $n\geq 2$, and gives the following version of Minkowski first theorem: {\it Let $K \subset \R^n$ be a symmetric convex body such that $\text{int}(K) \cap \Z^n = \set{0}$. Then its Lebesgue measure satisfies $\abs{K} \leq 2^n$.}

Let us now explain how this theorem translates into an optimal isosystolic inequality.
A symmetric convex body $K \subset \R^2$ corresponds\footnote{Namely set $\|v \|_{K}:=\inf\{t>0 \mid v \in tK\}$ for any $v\in \R^2$.} to a unique symmetric norm $\|\cdot \|_{K}$ on $\R^2$, which induces a unique $\Z^2$-periodic flat Finsler reversible metric $\tilde{F}_{K}$ on $\R^2$ by setting $\tilde{F}_{K}(\tilde{x},\tilde{v}):=\|\tilde{v}\|_{K}$. In this way to each convex body $K$ corresponds a unique flat Finsler reversible metric $F_K$ on the $2$-torus $\T^2$. Observe that given two points $\tilde{x}, \tilde{y} \in \R^2$, the length of any smooth curve $\gamma: [a,b] \rightarrow \R^2$ from $\tilde{x}$ to $\tilde{y}$ satisfies
\begin{equation*}
    \ell_{\tilde{F}_{K}}(\gamma) = \int_a^b \|\dot{\gamma}(t)\|_{K} dt \geq \left\|\int_a^b \dot{\gamma}(t) dt\right\|_{K} = \|\gamma(b) - \gamma(a)\|_{K} = \|\tilde{y}-\tilde{x}\|_{K}.
\end{equation*}
 Equality occurs when velocity is a constant vector, that is when $\gamma$ suitably parametrizes a line segment.
Now, any non-contractible closed curve of $\T^2$ lifts to $\R^2$ to a curve between two points $\tilde{x}$ and $\tilde{x}+z$ for some $z \in \Z^2\setminus\{0\}$. The length of such a curve is thus at least equal to  $\norm{z}_K$ from which we deduce that
\begin{equation*}
    \sys{(\T^2,F_K)}  = \min_{z \in \Z^2 \backslash \set{0}} \norm{z}_K.
\end{equation*}
Therefore 
\begin{align*}
        \text{int}(K) \cap \Z^2 = \set{0} &\Longleftrightarrow  \sys{(\T^2,F_K)} \geq 1,
    \end{align*}
    while using Definition \ref{def:area} we get that
\begin{align*}
        \abs{K} \leq 4 &\Longleftrightarrow \area_{BH}{(\T^2, F_K)} = \int_{[0,1]^2} \frac{\pi}{\abs{K}} d\tilde{x}_1 \wedge d\tilde{x}_2 = \frac{\pi}{\abs{K}} \geq \frac{\pi}{4}.
    \end{align*}
    
 As the systolic area remains invariant under rescaling the metric by any positive factor $\lambda$, and observing the fact that $\sys{(\T^2,\lambda F_K)} =\lambda     \sys{(\T^2,F_K)}$, we can reformulate Theorem \ref{th:firstMin} as the following optimal isosystolic inequality for Busemann-Hausdorff area and flat Finsler reversible metrics on the $2$-torus.

\begin{theorem}[Minkowski's first theorem, systolic formulation]\label{th:Mink}
Any flat Finsler reversible torus $(\T^2,F_K)$ satisfies the following optimal isosystolic inequality:
    \begin{equation*}
        \area_{BH}{(\T^2,F_K)} \geq \frac{\pi}{4} \sys^2{(\T^2,F_K)}.
    \end{equation*}
 Equality holds for the flat metric corresponding to the supremum norm $\|\cdot\|_\infty$.
 \end{theorem}

\subsection{Busemann-Hausdorff area in the flat non-reversible case: systolic freedom}\label{sec:sysfree}

It is well known that Minkowski first theorem no longer holds if we relax the symmetry assumption on the convex body. Equivalently, this means that there does not exist an isosystolic inequality on the $2$-torus for the Busemann-Hausdorff area and flat Finsler (possibly non-reversible) metrics. 

More especifically, consider for every $\varepsilon \in (0,1)$ the convex body $K_\varepsilon\subset \R^2$    in Figure \ref{fig:rhombus}  
        defined as the convex hull of the four vertices
$(0,1), (\tfrac{1+\varepsilon}{2\varepsilon},\tfrac{1-\varepsilon}{2}), (0,-\varepsilon)$ and $(-\tfrac{1+\varepsilon}{2\varepsilon}, \tfrac{1-\varepsilon}{2})$.

\begin{figure}[h!]
\includegraphics[scale=1]{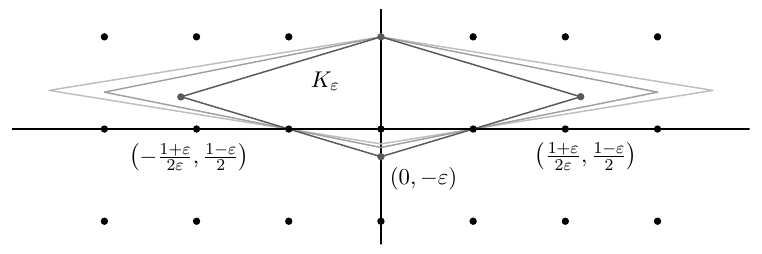}
      \caption{The convex body  $K_\varepsilon$.}
      \label{fig:rhombus}
\end{figure}

This convex body defines a flat Finsler metric $F_{K_\varepsilon}$ on $\T^2$ which is not reversible.
We easily check that $\sys(\T^2,F_{K_\varepsilon})=1$ and
$\abs{K_\varepsilon} = (1+\varepsilon)^2/(2\varepsilon).$
Therefore  its Busemann-Hausdorff systolic area is not bounded from below:
$$
\frac{\area_{BH}(\T^2,F_{K_\varepsilon})}{\sys^2(\T^2,F_{K_\varepsilon})} = \frac{2\pi\varepsilon}{(1+\varepsilon)^2} \xrightarrow[]{\varepsilon\to 0} 0.
$$

\subsection{Holmes-Thompson area in the flat reversible case}
We now focus on Holmes-Thompson notion of area in the reversible case. First recall the following optimal inequality.

\begin{theorem}\cite{Mahler2}\label{th:Mahler}
Given a symmetric convex body $K\subset \R^2$, the following inequality holds true:
$$
|K|\cdot |K^\circ|\geq 8.
$$
 Equality holds when $K$ is the unit disc of the supremum norm $\|\cdot\|_\infty$.
\end{theorem}

\begin{proof}[Sketch of proof.]
	We briefly describe here the strategy behind the proof of Mahler's inequality, and we refer to \cite{henze} for technical details. Any symmetric convex body can be approximated in the Hausdorff topology by a sequence of symmetric polygons. Since the product area $K \mapsto |K|\cdot |K^\circ|$ is continuous on $\mathcal{K}_0(\R^2)$, it suffices to prove the inequality for symmetric polygons. 
	
	It can be shown that given a symmetric polygon $P$ with $m \geq 3$ pairs of opposite vertices, one can construct a symmetric polygon $Q$ with $m-1$ pairs of opposite points such that $\abs{Q} \cdot \abs{Q^\circ} \leq \abs{P} \cdot \abs{P^\circ}$. Namely, set $P=\text{conv}\{\pm v_1,\ldots,\pm v_m\}$ and fix three adjacent vertices $v_{i-1},v_i$ and $v_{i+1}$ of $P$ whose convex hull does not contain the origin (using a cyclic notation). 
	Consider the line $L$ through $v_i$ and parallel to the line $(v_{i-1} \, \,v_{i+1})$. Moving continuously $v_i$ along $L$ preserves the volume of $P$, while convexity is ensured until $v_i$ reaches one of the two intersection points of $L$ with the line $(v_{i-2} \, \,v_{i-1})$ or $(v_{i+1}\, \, v_{i+2})$. By convexity, it can be shown that $\abs{P^\circ}$ is precisely minimal at one of these two points where the deformed polygon $P$ becomes a new polygon $Q$ with $m-1$ pairs of opposite vertices. Applying this result successively, one can reduce any symmetric polygon to a centered parallelogram while decreasing the product area. For such a parallelogram it is easy to check that area product is equal to $8$.
\end{proof}

Mahler also conjectured in \cite{mahler1} that in arbitrary dimension $n\geq 2$ the volume product of a symmetric convex body $K\subset \R^n$ satisfies the following lower bound:
$$
|K|\cdot |K^\circ|\geq \frac{4^n}{n!}.
$$
Mahler's conjecture has been recently proved in dimension $3$ by \cite{IS20}. In arbitrary dimension, the best known lower bound is due to \cite{kuperberg} who proved that
$$
|K|\cdot |K^\circ|\geq  \frac{\pi^n}{n!}.
$$
By combining Minkowski's first theorem together with Mahler's theorem, we obtain the following optimal isosystolic inequality for Holmes-Thompson area and flat Finsler reversible metrics on the $2$-torus.

\begin{theorem}\label{th:HT_reversible}
Any flat reversible Finsler torus $(\T^2,F_K)$ satisfies the following optimal isosystolic inequality:
    \begin{equation*}
        \area_{HT}{(\T^2,F_K)} \geq \frac{2}{\pi} \sys^2{(\T^2,F_K)}.
    \end{equation*}
 Equality holds for the flat metric corresponding to the supremum norm $\|\cdot\|_\infty$. \end{theorem}

\begin{proof}
Rescaling the metric if necessary, we can suppose that $\sys{(\T^2,F_K)}=1$. Now Minkowski first Theorem \ref{th:firstMin} ensures that $|K|\leq 4$ which together with Mahler's Theorem \ref{th:Mahler} implies that 
$$
 |K^\circ|\geq 2 \Leftrightarrow  \area_{HT}{(\T^2,F_K)}\geq \frac{2}{\pi}.
$$
\end{proof}

\subsection{Holmes-Thompson area in the flat non-reversible case}
We now present the optimal isosystolic inequality for Holmes-Thompson area and flat Finsler metrics on the two-torus obtained in \cite{ABT}.

\begin{theorem}\label{th:HTnonrev}
    \label{th:ABT}
Any flat Finsler torus $(\T^2,F_K)$ satisfies the following optimal isosystolic inequality:
    \begin{equation*}
        \area_{HT}{(\T^2,F_K)} \geq \frac{3}{2\pi} \sys^2{(\T^2,F_K)}.
    \end{equation*}
    Equality holds for the flat metric corresponding to the norm on $\R^2$ whose unit disc is the triangle with vertices $(1,0)$, $(0,1)$ and $(-1,-1)$.
\end{theorem}

\begin{proof}[Sketch of proof.]
We present here a short version of the proof, focusing on the main geometric ideas and avoiding several technical considerations.

First we bring the above isosystolic inequality into the world of the geometry of numbers as follows. As it is enough to show that if  $\sys{(\T^2,F_K)}\geq 1$ then $\area_{HT}{(\T^2,F_K)} \geq \frac{3}{2\pi}$, we have to prove that for a convex body $K\subset \R^2$ the condition $\text{int}(K) \cap \Z^2 = \set{0}$ ensures that $|K^\circ|\geq 3/2$. Now observe that  $\text{int}(K) \cap \Z^2 = \set{0}$  if and only if every integer line $m_1\tilde{x}_1+m_2\tilde{x}_2=1$ where $(m_1,m_2)\in \Z^2\setminus\{0\}$ intersects $K^\circ$.

\begin{figure}[h!]
\includegraphics[scale=0.6]{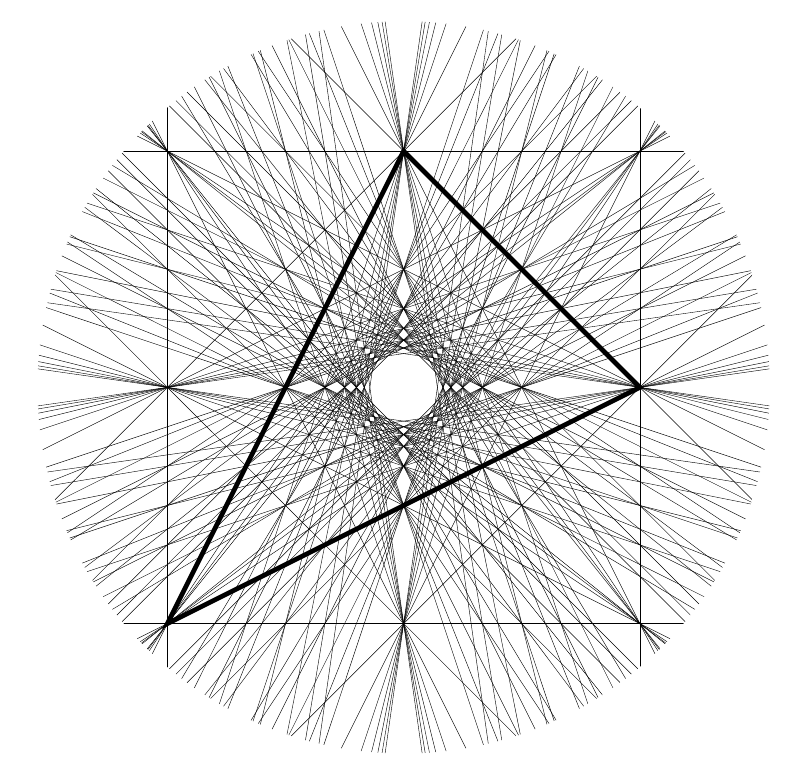}
      \caption{Set of integer lines     not parallel to the axes} for $m_1^2+m_2^2\leq 50$ and a triangle with minimal area.
      \label{integerlines}
       \end{figure}

So we are left to prove the following assertion: {\it if a convex body $Q\subset \R^2$ intersects every integer line, then its Lebesgue measure satisfies $\abs{Q}\geq 3/2$}.
If you wonder how this set of integer lines looks like, see Figure \ref{integerlines}. The convex body whose boundary is the bold triangle with vertices $(1,0)$, $(0,1)$ and $(-1,-1)$ in Figure \ref{integerlines} has area $3/2$ and intersects every integer line, showing that the isosystolic inequality in Theorem \ref{th:HTnonrev} is optimal.

Now we argue as follows. By approximation of convex bodies by polygons in the Hausdorff topology, it is enough to prove the assertion for a convex polygon $Q$ whose vertices will be denoted by $v_1,\ldots,v_n$. Remark that necessarily $n\geq 3$. We will argue  by induction on the number $n$ of vertices.  

If some vertex $v_i$ does not lie on an integer line supporting $Q$, fix any supporting line $\ell$ passing through $v_i$. 
There exists along $\ell$ at least one direction in which moving the vertex $v_i$ does not increase the area of the deformed polygon, see Figure \ref{fig:move}. Suppose that the correct direction corresponds to the adjacent vertex $v_{i-1}$. 
We can push\footnote{The crucial point here is that the set of integer lines not intersecting a small closed disc around the origin is always finite, so the set of integer lines that are candidates to appear as a new supporting line for the deformed polygon is also always finite.}  the vertex $v_i $ along $\ell$ in this direction until either $v_i$ meets for the first time an integer line supporting the deformed polygon, or $v_i$ meets the line $(v_{i-2},v_{i-1})$ and thus $v_{i-1}$ disappears as a vertex.
This deformation does not increase the area of the polygon while preserving convexity and the property of intersecting every integer line.
Applying this process to each vertex that is not contained in an integer line supporting $Q$, we deform our original  convex polygon into a new one (still denoted by $Q$) with at most the same area, that still meets every integer line, and such that every vertex is contained in at least one supporting integer line. The number of vertices may have decreased during the process.

\begin{figure}[h!]
\includegraphics[scale=1]{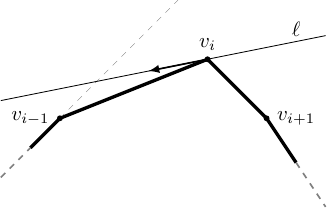}
      \caption{Deformation of the convex polygon.}
      \label{fig:move}
\end{figure}

Now suppose that a vertex $v_i$ is contained in exactly one integer line $\ell$ supporting $Q$. We argue exactly as above. More precisely, we push the vertex $v_i$ along $\ell$ in a direction that ensures the area does not increase, and stop either when the vertex $v_i$ meets another integer line supporting $Q$, or when one of the vertices adjacent to $v_i$ disappears as a vertex of $Q$.
Applying this process to each vertex contained in exactly one integer line supporting $Q$, we deform the convex polygon obtained in the previous step into a new one  (still denoted by $Q$) with at most the same area, that still meets every integer line, and such that every vertex is contained in at least two distinct supporting integer line. The number of vertices may have decreased during the process.
        
We now show that it implies that every vertex of our new convex polygon still denoted by $Q$ actually belongs to $\Z^2$. For this, fix a vertex $v$ contained in at least two distinct integer lines. As each integer line is written as a set of the form $\{\langle \xi,\cdot\rangle=1\}$ for some $\xi \in \Z^2$, we can find in particular two supporting integer lines $L_1,L_2$ containing $v$  whose associated lattice points $\xi_1,\xi_2$ are such that the interior of the segment $[\xi_1,\xi_2]$ does not intersect $\Z^2$. Both $\xi_1$ and $\xi_2$ automatically belong to the boundary of $Q^\circ$, and therefore the triangle formed by the origin together with these two integral points  does not contain any other integral points than its vertices as $\text{int}(Q^\circ)\cap \Z^2=0$ by assumption. It ensures that $(\xi_1,\xi_2)$ forms a basis of $\Z^2$, and thus the associated matrix $A$ with row vectors $\xi_1$ and $\xi_2$ belongs to $SL_2(\Z)$. Next observe that $v=(x_1,x_2)$ is the unique solution of the equation $Av^t=(1,1)^t$. Thus $v=A^{-1}(1,1)^t\in \Z^2$.

To finish the proof, we apply Pick's formula \cite{Pick} that asserts that for a convex polygon $Q$ whose vertices belong to $\Z^2$, if we denote by $i$ the number of integer points contained in its interior, and by $b$ the number of integer points contained in its boundary, the area of the polygon satisfies $|Q|=i+b/2-1$. In our case,  as the origin of the plane belongs to the interior of our polygon which has at least $3$ integer vertices, we have $i\geq 1$ and $b\geq 3$ which ensures that $|Q|\geq 3/2$.

The triangle with vertices $(-1,-1)$, $(0,1)$ and $(1,0)$ pictured in Figure \ref{integerlines}  has area $3/2$. It also intersects every integer line as its polar is the triangle with vertices $(1,1)$, $(1,-2)$ and $(-2,1)$ whose interior does not contain other integer points than the origin. It ensures the optimality of our assertion, and therefore of the corresponding isosystolic inequality.
\end{proof}

\section{Asymptotic geometry in the universal cover of a Finsler $2$-torus}\label{sec:stable_norm}

In sections \ref{sec:HT} and \ref{sec:BH}, we will prove several optimal isosystolic inequalities on (non-flat) Finsler $2$-tori. The same strategy applies to all these proofs and is similar to the Riemannian case: find a flat metric whose systolic area is smaller than the original metric, and then apply the corresponding flat optimal systolic inequality already proved in section \ref{sec:flatFinsler}. The flat metric here will be defined using the asymptotic geometry of the universal covering space of our Finsler torus, and is known as the {\it stable norm}. In this section we first present this notion in subsection \ref{sec:stab} and prove that passing for a $2$-torus from a Finsler metric to its stable norm does not change the systole, see Proposition \ref{prop:sys_stab}. Therefore the proof of these optimal isosystolic inequalities will reduce to prove that passing from a Finsler metric to its stable norm decreases the corresponding notion of area. For this we will need a technical tool introduced in \cite{burago_ivanov} and called {\it calibrating functions}. These functions---defined in analogy with Busemann functions---are presented in subsection \ref{sec:cal} together with the main properties we will need in the forthcoming sections. \\

\subsection{Stable norm}\label{sec:stab}

We closely follow here the presentation used in \cite{burago_ivanov}. 
We refer  to \cite{Fed74} the reader interested in the various alternative definitions of the stable norm and their equivalence. Fix a Finsler metric $F$ on $\T^2$ inducing a $\Z^2$-periodic Finsler metric $\tilde{F}$ on its universal cover $\R^2$, and a point $\tilde{x}_0 \in \R^2$.
Then for any $z \in \Z^2$ set
\begin{equation}\label{def:stab}
\|z\|^F_{st}:=\lim_{k\to \infty} \frac{d_{\tilde{F}}(\tilde{x}_0,\tilde{x}_0+kz)}{k}.
\end{equation}
It is well known that this limit exists, does not depend on the choice of $\tilde{x}_0$, and that the function $\|\cdot\|_{st}^F$ extends to a norm on $\R^2$ called the {\it stable norm}.
The bounded distance theorem (see \cite{burago}) states that the stable norm of $F$ turns out to be the unique norm on $\R^2$ such that there exists a constant $C$ for which
\begin{equation}\label{eq:st}
\|z\|^F_{st}\leq d_{\tilde{F}}(\tilde{x},\tilde{x}+z)\leq \|z\|^F_{st}+C
\end{equation}
for every $z\in \Z^2$ and any $\tilde{x} \in \R^2$.
So, informally speaking, the stable norm naturally appears by looking at the distance function on the universal cover of the Finsler torus at a large scale.
By passing to the quotient, the stable norm induces a Finsler flat metric on $\T^2$ still called stable norm and denoted by $\|\cdot\|_{st}^F$. It is worth noting that $\|\cdot\|_{st}^{F}=F$ if and only if $F$ is flat. 

The following result will be of fundamental importance to us:
\begin{proposition}\label{prop:sys_stab}
$\sys(\T^2,F)=\sys(\T^2,\|\cdot\|^F_{st})$.
\end{proposition}

\begin{proof}
Since $\|\cdot\|_{st}^F$ depends continously on $F$ (see \cite{burago_ivanov}), we may assume that $F:T\T^2\to \R$ is smooth outside the zero section and quadratically convex (that is, the second derivatives of $F^2_{|T_x\T^2\setminus\{0\}}$ are positive definite for all $x\in \T^2$). 
Then, according to Lemma 4.32 in \cite[p.260]{Gromov} which still holds for smooth and quadratically convex Finsler metrics, if a closed curve is length minimizing in its homotopy class, its iterates are also minimizing in their respective homotopy classes. Therefore
$$
\|z\|^F_{st}=\min_{\tilde{x}\in \Delta} d_{\tilde{F}}(\tilde{x},\tilde{x}+z)
$$ 
for any fundamental domain $\Delta$ in $\R^2$  for the $\Z^2$-action, and the conclusion easily follows.
\end{proof}

\subsection{Dual stable norm and calibrating functions}\label{sec:cal}

We now present the notion of calibrating functions. All the material of this subsection can be found  in \cite{burago_ivanov}. From now on, we assume the Finsler metric $F$ to be smooth and quadratically convex, and we fix an arbitrary point  $\tilde{x}_0 \in \R^2$. Denote by $\|\cdot\|_{st}$ its stable norm and by $\|\cdot\|_{\tilde{x}}$ the norm defined on each tangent space $T_{\tilde{x}}\R^2\simeq \R^2$ by the Finsler metric $\tilde{F}$.  For a linear form $h:\R^2\to\R$, we define its dual stable norm  by $\|h\|_{st}^\ast:=\max\{h(v)\mid \|v\|_{st}\leq 1\}$ and its dual Finsler norm at $\tilde{x}$ by $\|h\|_{\tilde{x}}^\ast:=\max\{h(v)\mid \|v\|_{\tilde{x}}\leq 1\}$.
Define by
$$
K^\ast_{\tilde{x}}:=\{h \in (\R^2)^\ast \mid \|h\|^\ast_{\tilde{x}}\leq1\} \, \, \text{and} \, \, 
B_{st}^\ast:=\{h \in (\R^2)^\ast \mid \|h\|^\ast_{st}\leq1\}
$$
the corresponding dual unit balls. 

\begin{lemma}\label{lem:cal}
Let $h\in \partial B^\ast_{st}$. The function
\begin{equation*}
	f(\tilde{x}): = \limsup_{\Z^2 \ni z\to\infty} \left[h(z) - d_{\tilde{F}}(\tilde{x},\tilde{x}_0 +z)\right]
\end{equation*}
is well defined and satisfies the following properties:
 
 \begin{enumerate}
		\item $f(\tilde{x}+z) = f(\tilde{x}) + h(z)$ for all $\tilde{x} \in \R^2$ and $z \in \Z^2$.
        \item $d_{\tilde{x}} f$ is defined for almost every point $\tilde{x} \in \R^2$ and satisfies $\|d_{\tilde{x}} f\|_{\tilde{x}}^\ast=1$.
\end{enumerate}
\end{lemma}

Such a function is an example of calibrating function for $h$ (see \cite{burago_ivanov} for a precise definition) and is defined in analogy with Busemann functions.

\begin{proof}[Proof of the Lemma \ref{lem:cal}]
Using the fact that $h(z)\leq \|z\|_{st}$ as $\|h\|_{st}^\ast=1$, and formula (\ref{eq:st}), we see that $f(\tilde{x})\leq d_{\tilde{F}}(\tilde{x},\tilde{x}_0)<+\infty$. Besides, we can always find a sequence $\{z_i\}$ of points in  $\Z^2$ such that $z_i \to \infty$ and $h(z_i)\geq \|z_i\|_{st}-c$ for some constant $c$. For this, first observe that since $\|h\|_{st}^\ast=1$ we can find a vector $v\in \partial B_{st}$ such that $h(v)=1$. Then, consider the fundamental domain $\Delta:=[0,1]^2\subset\R^2$  for the $\Z^2$-action, and denote by $\{\Delta_i:=\Delta+z_i\}$ the sequence of fundamental domains translated by some element $z_i\in \Z^2$ successively intersecting the ray $\{tv\mid t>0\}$. It is easy to see that the sequence $\{z_i\}$ is suitable. Now, using (\ref{eq:st}) again, we see that $f(\tilde{x})\geq -d(\tilde{x},\tilde{x}_0)-C-c>-\infty$. So $f$ always takes finite values and is therefore well defined.

We easily check that	
\begin{eqnarray*}
f(\tilde{x}+z)& =& \limsup_{\Z^2 \ni z'\to\infty} \left[h(z') - d_{\tilde{F}}(\tilde{x}+z,\tilde{x}_0 +z')\right]\\
&=& \limsup_{\Z^2 \ni z'\to\infty} \left[h(z'+z) - d_{\tilde{F}}(\tilde{x}+z,\tilde{x}_0 +z'+z)\right]\\
&=& f(\tilde{x})+h(z),
\end{eqnarray*}
so property (1) holds.

Next observe that $f$ is $1$-Lipschitz  with respect to the Finsler metric $\tilde{F}$ as an upper limit of $1$-Lipschitz functions:
$$\abs{f(\tilde{x})-f(\tilde{y})} \leq d_{\tilde{F}}(\tilde{x},\tilde{y})$$ for all $\tilde{x},\tilde{y} \in \R^2$.
So its differential $d_{\tilde{x}} f$ is defined for almost every point $\tilde{x} \in \R^2$ and satisfies $\|d_{\tilde{x}} f\|_{\tilde{x}}^\ast\leq1$. In order to prove the reverse inequality, we will prove that for every $\tilde{x} \in \R^2$ there is a geodesic ray $\eta: [0,\infty) \rightarrow (\R^2,\tilde{F})$ with origin $\eta(0) = \tilde{x}$ satisfying $f(\eta(t)) \geq f(\tilde{x}) + t$.
Indeed fix the point $\tilde{x}\in \R^2$ and choose a sequence $z_i\to \infty$ of points in  $\Z^2$ such that 
$$
	f(\tilde{x}) = \lim_{i\to\infty} \left[h(z_i) - d_{\tilde{F}}(\tilde{x},\tilde{x}_0 +z_i)\right].
$$
Denote by $\eta_i$ a minimal geodesic path going from $\tilde{x}$ to $\tilde{x}_0 +z_i$ parametrized by arc length. By compactness, we can find a converging subsequence $\{\eta_i'(0)\}\subset S_{\tilde{x}}$ to some vector $v\in S_{\tilde{x}}$. This defines a geodesic ray $\eta$ starting at $\tilde{x}$ by setting $\eta'(0)=v$. Now
\begin{eqnarray*}
f(\eta(t))&=&\limsup_{\Z^2 \ni z\to\infty} \left[h(z) - d_{\tilde{F}}(\eta(t),\tilde{x}_0 +z)\right]\\
&\geq&\limsup_{i\to\infty} \left[h(z_i) - d_{\tilde{F}}(\eta(t),\tilde{x}_0 +z_i)\right]\\
\text{(by pointwise convergence)}&=&\limsup_{i\to\infty} \left[h(z_i) - d_{\tilde{F}}(\eta_i(t),\tilde{x}_0 +z_i)\right]\\
(\eta_i \, \text{minimal geodesic)}&=&\limsup_{i\to\infty} \left[h(z_i) - (d_{\tilde{F}}(\eta_i(0),\tilde{x}_0 +z_i)-d_{\tilde{F}}(\eta_i(0),\eta_i(t)))\right]\\
&=&\limsup_{i\to\infty} \left[h(z_i) - d_{\tilde{F}}(\tilde{x},\tilde{x}_0 +z_i)\right]+t\\
&=&f(\tilde{x})+t.
\end{eqnarray*}
Therefore $f(\eta(t)) \geq f(\tilde{x}) + t$ which implies the reverse inequality: $\|d_{\tilde{x}} f\|_{\tilde{x}}^\ast\geq1$. So $\|d_{\tilde{x}} f\|_{\tilde{x}}^\ast=1$ and property (2) is proved.

Finally, although we will not need this fact,  observe that $f(\eta(t))-f(\tilde{x})=f(\eta(t))-f(\eta(0))\leq t$ as $f$ is $1$-Lipschitz and $\eta$ is a geodesic parametrized by arc length, so the above inequality is in fact an equality: $f(\eta(t)) = f(\tilde{x}) + t$.
\end{proof}

\begin{remark}
    Notice that the $h(\Z^2)$-equivariance of the calibrating function $f$ implies that its differential $d_{\tilde{x}} f$ is $\Z^2$-periodic, and hence induces a $1$-form on the quotient $\T^2 = \R^2 / \Z^2$, that we abusively denote by $d_x f$, where $x = \pi(\tilde{x})$.
\end{remark}

The following technical result will be needed to prove that passing from a Finsler metric to its stable norm decreases the Holmes-Thompson area in the general case, and the Busemann-Hausdorff area in the reversible case.

\begin{lemma} \label{le:integral_forms}
Let $h,h' \in \partial B^\ast_{st}$ be two pairwise linearly independent linear forms, and $f,f'$  the associated calibrating functions defined using Lemma \ref{lem:cal}. The following holds true:
    \begin{equation*}
        \int_{\T^2} h \wedge h' = \int_{\T^2} d_x f \wedge d_x f'.
    \end{equation*}
\end{lemma}

\begin{proof}[Proof of the Lemma \ref{le:integral_forms}]
We can suppose that the dual basis $(h,h')$ is positively oriented.
    Consider the maps
    \begin{eqnarray*}
	      L : \R^2  \to    \R^2   \, \, \, \, \, \, \, \,  \, \, \, \, \, \, \, \, \, \, \, \, \, \, \, \, & \, \, \text{and} \, \, &     G : \R^2 \to    \R^2\\
	      \tilde{x}  \mapsto    (h(\tilde{x}), h'(\tilde{x}))   	    & &   \, \, \, \, \, \, \, \,  \, \, \, \, \, \,  \tilde{x}  \mapsto    (f(\tilde{x}), f'(\tilde{x}))   
    \end{eqnarray*}
    The linear map $L$ induces a map of degree one on the quotient
    \begin{equation*}
        \bar{L} : \T^2=\R^2/\Z^2 \to \R^2/L(\Z^2).
    \end{equation*}
    Besides, according to Lemma \ref{lem:cal}, for any $\tilde{x} \in \R^2$ and $z \in \Z^2$, we have that
$G(\tilde{x}+z)=   G(\tilde{x})+L(z)$,
    and therefore each map $G$ also induces a well defined map
    \begin{equation*}
        \bar{G} : \T^2=\R^2/\Z^2 \to \R^2/L(\Z^2)
    \end{equation*}
    of degree one.
    The linear forms  $d\tilde{x}_1$ and $d\tilde{x}_2$ on $\R^2$ induce $1$-forms on $\R^2/L(\Z^2)$ that we denote by $dx_1$ and $dx_2$. Then
\[        \int_{\T^2} h \wedge h' = \int_{\T^2} \bar{L}^\ast (dx_1\wedge dx_2) = \int_{\R^2/L(\Z^2)} dx_1 \wedge dx_2 = \int_{\T^2} \bar{G}^\ast (dx_1\wedge dx_2) = \int_{\T^2} d_x f \wedge d_x f'.
  \]
  \end{proof}

Finally we state the following intuitive result, a proof of which can be found in \cite[Lemma 5.1]{burago_ivanov}.
\begin{lemma}	\label{le:cyclic_order}
	Let $h_1, h_2, h_3 \in \partial B^\ast_{st}$ be three linear forms, and $f_1 , f_2 , f_3$ the three associated functions defined using Lemma \ref{lem:cal}. Let $\tilde{x} \in \R^2$ be a point where $d_{\tilde{x}} f_i$ is defined for $i=1,2,3$. 
	
	Then $\{d_{\tilde{x}} f_1 , d_{\tilde{x}} f_2 , d_{\tilde{x}} f_3\} \subset \partial B^\ast_{\tilde{F}}(\tilde{x})$ have the same cyclic order that $\{h_1 , h_2 , h_3\} \subset \partial B^\ast_{st}$.
\end{lemma}

Loosely speaking, this lemma holds true because two minimal geodesics intersect at most once. This fact is purely $2$-dimensional, and does no longer hold in higher dimensions.

\section{Isosystolic inequalities for Finsler metrics and Holmes-Thompson area}\label{sec:HT}

In this  section, we prove two optimal isosystolic inequalities on the $2$-torus for the Holmes- Thompson notion area: one for reversible Finsler metrics, and another for (possibly) non-reversible Finsler ones. As already said, the same strategy applies for both proofs and proceeds as follows: prove that the associated stable norm has smaller systolic area than the original metric, and then apply the corresponding flat optimal systolic inequality already proved in section \ref{sec:flatFinsler}. By Proposition \ref{prop:sys_stab}, the main step in the proof therefore boils down to show the following statement: {\it while passing for a 2-torus from a Finsler metric to its stable norm does not change the systole, it decreases the Holmes-Thompson area}.

\subsection{Decreasing the Holmes-Thompson area}

Let us formally state this as follows.

\begin{theorem} \cite{burago_ivanov} \label{th:decreasing_area}
	Let $(\T^2,F)$ be a Finsler torus. Then:
$$
		\area_{HT}{(\T^2,F)} \geq \area_{HT}(\T^2,\|\cdot\|_{st}^F).
$$
\end{theorem}

\begin{proof}
Since $\|\cdot\|_{st}^F$ depends continously on $F$, we assume the Finsler metric to be smooth and quadratically convex. The general statement follows by approximation. 

 Let $h_1, \dots h_N \in \partial B^\ast_{st}$ be a collection of cyclically positively ordered and pairwise linearly independent linear forms, and $f_1,\ldots,f_N$  the associated functions defined using Lemma \ref{lem:cal}.  By Lemma \ref{le:integral_forms} we know that
    \begin{equation*}
        \int_{\T^2} h_i \wedge h_{i+1} = \int_{\T^2} d_x f_i \wedge d_x f_{i+1}
    \end{equation*}
for $i=1,\ldots,N$ using a cyclic notation.
Now notice that $h_i \wedge h_{i+1} = \norm{h_i \wedge h_{i+1}}^\ast dx_1 \wedge dx_2$, where the norm of this 2-form is the one induced by the standard Euclidean structure of $\R^2$. Geometrically, the norm of such a 2-form coincides with twice the dual Lebesgue measure of the triangle defined by $0, h_i , h_{i+1}$ on $(\R^2)^*$, and since the collection of linear forms $h_1, \dots, h_N$ is cyclically ordered, this implies that
\begin{equation*}
    \frac{1}{2}\sum_{i = 1}^N h_i \wedge h_{i+1} = \abs{\conv{\set{h_i}_{i=1}^N}}^\ast dx_1 \wedge dx_2.
\end{equation*}
Here we have denoted by $\abs{\cdot}^\ast$  the dual Lebesgue measure on $(\R^2)^*$.
Similarly, and because the collection $d_x f_1, \dots d_x f_N$ is also cyclically ordered according to Lemma \ref{le:cyclic_order}, we can write
\begin{equation*}
    \frac{1}{2} \sum_{i = 1}^N d_x f_i \wedge d_x f_{i+1} = \abs{\conv{\set{d_x f_i}_{i=1}^N}}^\ast dx_1 \wedge dx_2.
\end{equation*}
Thus Lemma \ref{le:integral_forms} implies that
\begin{equation*}
    \int_{\T^2} \abs{\conv{\set{h_i}_{i=1}^N}}^\ast dx_1 \wedge dx_2 = \int_{\T^2} \abs{\conv{\set{d_x f_i}_{i=1}^N}}^\ast dx_1 \wedge dx_2.
\end{equation*}

Using the isomorphism $\R^2 \simeq (\R^2)^\ast$ given by $\tilde{x} \mapsto \langle \tilde{x},\cdot\rangle$, we can check that given a convex body $K$ the Lebesgue measure of the polar body $K^\circ$ and the dual Lebesgue measure of the dual body $K^\ast:=\{f \in (\R^2)^\ast \mid f|_{K}\leq 1\}$ coincide.
Therefore, for each choice of $\{h_i\}_{i=1}^N$ we can write the Holmes-Thompson area of $(\T^2,\|\cdot\|_{st}^F)$ as
\begin{align*}
    \area_{HT}(\T^2,\|\cdot\|_{st}^F) &= \frac{\abs{B^\ast_{st}}^\ast}{\pi} \\
    &= \frac{\abs{B^\ast_{st}}^\ast}{\abs{\conv{\set{h_i}}}^\ast} \cdot \frac{1}{\pi} \int_{\T^2} \abs{\conv{\set{h_i}}} ^\ast dx_1 \wedge dx_2 \\
    &= \frac{\abs{B^\ast_{st}}^\ast}{\abs{\conv{\set{h_i}}}^\ast} \cdot \frac{1}{\pi} \int_{\T^2} \abs{\conv{\set{d_x f_i}}}^\ast dx_1 \wedge dx_2 \\
    \text{($\conv{\set{d_x f_i}}\subset K_x^\ast$, by convexity)} \quad &\leq \frac{\abs{B^\ast_{st}}^\ast}{\abs{\conv{\set{h_i}}}^\ast} \cdot \frac{1}{\pi} \int_{\T^2} \abs{K^\ast_x}^\ast dx_1 \wedge dx_2   \\
    &= \frac{\abs{B^\ast_{st}}^\ast}{\abs{\conv{\set{h_i}}}^\ast} \area_{HT} (\T^2, F).
\end{align*}

Finally, an adequate choice of the collection $\set{h_i}_{i=1}^N$ allows as to make the ratio $\abs{B^\ast_{st}}^\ast / \abs{\conv{\set{h_i}}}^\ast$ arbitrarily close to 1. Hence, we conclude that
\begin{equation*}
    \area_{HT}(\T^2,\|\cdot\|_{st}^F) \leq \area_{HT} (\T^2, F).
\end{equation*}
  \end{proof}

\subsection{Holmes-Thompson area in the non-flat reversible case}

We now deduce the optimal isosystolic inequality for Holmes-Thompson area and reversible Finsler metrics on the two-torus first observed in \cite{sabourau}.

\begin{theorem}\cite{sabourau}
	Any Finsler reversible torus $(\T^2,F)$ satisfies the following optimal isosystolic inequality:
	\begin{equation*}
	\area_{HT}(\T^2,F) \geq \frac{2}{\pi} \sys{(\T^2,F)}^2.
	\end{equation*}	
 Equality holds for the flat metric corresponding to the supremum norm $\|\cdot\|_\infty$.
\end{theorem}

\begin{proof}
	Given a reversible Finsler metric on $\T^2$, we have that
	\begin{eqnarray*}
	\area_{HT}(\T^2,F) & \geq  & \area_{HT}(\T^2,\|\cdot\|_{st}^F) \, \, \text{(by Theorem \ref{th:decreasing_area})}\\
&\geq &{2\over \pi} \sys(\T^2,\|\cdot\|_{st}^F)^2 \, \, \text{(by Theorem \ref{th:HT_reversible})}\\
&= &{2\over \pi} \sys(\T^2,F)^2 \, \, \text{(by Proposition \ref{prop:sys_stab})}
\end{eqnarray*}
and the proof is complete.
\end{proof}

\subsection{Holmes-Thompson area in the non-flat and non-reversible case}

We establish now the optimal isosystolic inequality for Holmes-Thompson area and possibly non-reversible Finsler metrics on the two-torus. This result first appears in  \cite{ABT}.

\begin{theorem}\cite{ABT}
	Any Finsler torus $(\T^2,F)$ satisfies the following optimal isosystolic inequality:
	\begin{equation*}
		\area_{HT}(\T^2,F) \geq \frac{3}{2\pi} \sys(\T^2,F)^2.
	\end{equation*}
    Equality holds for the flat metric corresponding to the norm on $\R^2$ whose unit disc is the triangle with vertices $(1,0)$, $(0,1)$ and $(-1,-1)$.
\end{theorem}

\begin{proof}
		Given a Finsler metric $F$ on $\T^2$, we have that
	\begin{eqnarray*}
	\area_{HT}(\T^2,F) & \geq  & \area_{HT}(\T^2,\|\cdot\|_{st}^F) \, \, \text{(by Theorem \ref{th:decreasing_area})}\\
&\geq &{3\over 2 \pi} \sys(\T^2,\|\cdot\|_{st}^F)^2 \, \, \text{(by Theorem \ref{th:ABT})}\\
&= &{3\over 2 \pi} \sys(\T^2,F)^2 \, \, \text{(by Proposition \ref{prop:sys_stab})}
\end{eqnarray*}
and the proof is complete.
\end{proof}

\section{Isosystolic inequality for reversible Finsler metrics and Busemann-Hausdorff area}\label{sec:BH}

In this last section, we prove the optimal isosystolic inequality on the $2$-torus for the Busemann-Hausdorff notion of area and for reversible Finsler metrics. The main step consists in proving the following analog of Theorem \ref{th:decreasing_area} which was indirectly proved in \cite{burago_ivanov_2012}.

\begin{theorem} \cite{burago_ivanov_2012} \label{th:decreasing_area_BH}
	Let $(\T^2,F)$ be a reversible Finsler torus. Then:
$$
		\area_{BH}{(\T^2,F)} \geq \area_{BH}(\T^2,\|\cdot\|_{st}^F).
$$
\end{theorem}

Indeed using this result we can easily show the optimal isosystolic inequality presented in the introduction as follows.
\begin{proof}[Proof of Theorem \ref{th:opt}]
Given a Finsler reversible metric $F$ on $\T^2$, we have that
	\begin{eqnarray*}
	\area_{BH}(\T^2,F) & \geq  & \area_{BH}(\T^2,\|\cdot\|_{st}^F) \, \, \text{(by Theorem \ref{th:decreasing_area_BH})}\\
&\geq &{\pi \over 4} \sys(\T^2,\|\cdot\|_{st}^F)^2 \, \, \text{(by Minkowski's first Theorem \ref{th:Mink})}\\
&= &{\pi\over 4} \sys(\T^2,F)^2 \, \, \text{(by Proposition \ref{prop:sys_stab})}
\end{eqnarray*}
and the proof is complete.
\end{proof}

\subsection{Decreasing the Busemann-Hausdorff area}\label{sec:decreasing_area_BH}

So we are left to prove the above theorem.

\begin{proof}[Proof of Theorem \ref{th:decreasing_area_BH}]
As already explained in the introduction, it turns out that this result has been indirectly proved  in \cite{burago_ivanov_2012}. More precisely, the authors proved there that {\it a region in a two-dimensional affine subspace of a normed space has the least Hausdorff area among all compact surfaces with the same boundary}. But another result \cite[Theorem 1]{burago_ivanov} of the same authors ensures that this statement is equivalent to Theorem \ref{th:decreasing_area_BH}. 
We remark that their proof of this equivalence is by no means straightforward, and this is why we now propose a more direct proof of Theorem \ref{th:decreasing_area_BH} based on the notion of calibrating functions, and the following technical result.

\begin{proposition} \cite{burago_ivanov_2012} \label{pr:estimate_wedge_product}
    Let $K \subset \R^2$ be a symmetric convex body, 
    $\varphi_1 , \dots, \varphi_N \in K^*$ and $p_1, \dots, p_N$ a collection of nonnegative real numbers such that $\sum_{i=1}^N p_i= 1$. Then
    \begin{equation*}
        \sum_{1\leq i<j\leq N} p_i p_j \norm{\varphi_i \wedge \varphi_j}^\ast \leq \frac{1}{\abs{K}}
    \end{equation*}
    In addition, equality holds when $K$ is a symmetric convex $2N$-gon with (cyclically ordered) vertices $a_1 \dots a_{2N}$, $\varphi_i$ are the supporting functions $h_i$ associated to each edge $a_ia_{i+1}$ of $K$ and the weights are given by the formula $p_i = \norm{a_i \wedge a_{i+1}}/\abs{K}$.
\end{proposition}

\begin{remark}
Recall that the supporting function $h$ associated to an edge $e$ of  a polygon is the unique linear form such that the line $\{h=1\}$ contains the edge $e$.
    Also recall that $\norm{\cdot}$ (resp. $\norm{\cdot}^\ast$)  denotes here the norm induced by the Euclidean structure of $\R^2$ on the space of $2$-vectors (resp. the space of $2$-forms). In particular $\norm{a_i \wedge a_{i+1}}$ (resp. $\norm{\varphi_i \wedge \varphi_j}^\ast$) is twice the Lebesgue measure of the triangle defined by $0, a_i , a_{i+1}$ in $\R^2$ (resp. the dual Lebesgue measure of the triangle defined by $0, \varphi_i , \varphi_j$ in $(\R^2)^*$).
\end{remark}

\begin{proof}[Proof of Proposition \ref{pr:estimate_wedge_product}]
     We first prove the equality case. Assume that $K$ is a convex symmetric $2N$-gon with (cyclically ordered) vertices $a_1 \dots a_{2N}$, and take $p_i = \norm{a_i \wedge a_{i+1}}/\abs{K}$ for $i=1,\ldots,N$.
    If we set $v_i = \overrightarrow{a_ia_{i+1}}\in \R^2$ for $i=1\ldots N$, one can see that
  \[       
  \sum_{1\leq i<j\leq N} \norm{v_i \wedge v_j} = \sum_{j=2}^N \sum_{i=1}^j \norm{v_i \wedge v_j}
= \sum_{j=2}^N  \norm{\left(\sum_{i=1}^j v_i\right) \wedge v_j}
= \sum_{j=2}^N  \norm{\overrightarrow{a_1a_{j+1}} \wedge \overrightarrow{a_ja_{j+1}}}
       = \abs{K}
\]
since all the pairs $(v_i,v_j)$ for ${i<j}$ have the same orientation.
Denote for $i=1\ldots N$ by $h_i$ the supporting function associated to the edge $a_ia_{i+1}$ of $K$.  The map $J:v \in \R^2 \mapsto (dx_1\wedge dx_2)(v,\cdot)\in (\R^2)^\ast$ defines an isomorphism such that
     $$
     \norm{J(v)\wedge J(v')}^\ast=\norm{v \wedge v'}
     $$
     for any $v,v' \in \R^2$ as a simple computation shows.
We easily check that $J(v_i)=\norm{a_i \wedge a_{i+1}} h_i$, which  implies that
    \begin{equation*}
        p_i p_j \norm{h_i \wedge h_j} ^\ast= \frac{1}{\abs{K}^2}  \norm{v_i \wedge v_j}
    \end{equation*}
for any $1 \leq i<j\leq N$.
    We conclude by adding up these equalities.

    To prove the general inequality, we can reduce ourselves to the case where $K$ is a $2N$-gon and the forms $\varphi_i$ are precisely its supporting functions $h_1,\ldots,h_N$---that is, $K^\ast$ is the convex hull of the set $\{\pm h_1,\ldots,\pm h_N\}$---, while the weights $p_1, \dots, p_{N}$ are arbitrary. Indeed, first observe that it is enough to prove the result for symmetric polygons, as we can approximate $K$ by a sequence of symmetric polygons $\{K_i\}$ such that $K_i\subset K$ (this last condition ensures that $\varphi_1,\ldots,\varphi_N \in K_i^\ast$ as $K^\ast \subset K_i^\ast$  for all $i$, and thus proving the lemma for the $K_i$'s, with the original $\varphi_1,\ldots,\varphi_N$ and $p_1,\ldots,p_N$, will imply the general case). So let assume that $K$ is a convex symmetric $2n$-gon such that $K^\ast=\text{conv}\{\pm h_1,\ldots,\pm h_n\}$ where $n$ is possibly different from $N$.
    The sum $\sum_{1\leq i<j\leq N} p_i p_j \norm{\varphi_i \wedge \varphi_j}^\ast$ is convex in each variable $\varphi_i \in K^*$, hence its restriction to $K^*$ attains its maximum at some vertex of the boundary. Thus we can suppose that each $\varphi_i$ is a certain supporting function in $\{\pm h_1,\ldots,\pm h_N\}$ in order to bound this sum from above. If during this process we get $h_i = \pm h_j$ for some $i< j$, we can just drop one of these linear forms of the list, and replace the weight of the other linear form by the total weight $p_i + p_j$ without modifying the sum. At the end we find new weights $p_i'$ for $i=1,\ldots,n$ suming up to $1$ (by possibly setting $p'_i=0$ if the corresponding supporting function $h_i$ does not show up during the minimization process) such that 
    $$
    \sum_{1\leq i<j\leq N} p_i p_j \norm{\varphi_i \wedge \varphi_j}^\ast\leq\sum_{1\leq i<j\leq n} p'_i p'_j \norm{h_i \wedge h_j}^\ast.
    $$
Therefore in order to obtain the general inequality it is enough to prove that the second sum above is at most equal to $1/|K|$.

    Finally assume that $K$ is a symmetric convex $2N$-polygon with (cyclically ordered) vertices $a_1 \dots a_{2N}$ and supporting functions $h_i$  associated to each edge $a_ia_{i+1}$ for $i=1,\ldots,N$, and  let $p_1, \dots , p_N$ be nonnegative real numbers such that $\sum_i p_i = 1$. Set $v_i=\overrightarrow{a_ia_{i+1}}$, $q_i = \norm{a_i \wedge a_{i+1}}/\abs{K}$ and $\lambda_i = p_i/q_i$. Now consider a symmetric convex $2N$-gon $K'$ with vertices $a'_1 \dots a'_{2N}$ satisfying the condition $\overrightarrow{a'_ia'_{i+1}}=\lambda_i v_i$. Denote $v'_i=\lambda_i v_i$. Then
    \begin{eqnarray*}
        \sum_{i<j} p_i p_j \norm{h_i \wedge h_j}^\ast 
        &=& \sum_{i<j} \lambda_i \lambda_j q_i q_j \norm{h_i \wedge h_j}^\ast \\
        &= &\frac{1}{\abs{K}^2} \sum_{i<j} \lambda_i \lambda_j \norm{v_i \wedge v_j} \\ 
        &= &\frac{1}{\abs{K}^2} \sum_{i<j} \norm{v'_i \wedge v'_j} \\
        &= &\frac{\abs{K'}}{\abs{K}^2}. 
    \end{eqnarray*}
    Observe that the condition  $\sum_i p_i = 1$ is equivalent to the condition that $\abs{K}=V(K,K')$ where $V(K,K')$ denotes the mixed volume of $K$ and $K'$.  Indeed, denoting by $d_i$ the distance of the supporting line $\{h_i=1\}$ to the origin, we have
    \begin{eqnarray*}
    1&=&\sum_{i=1}^N p_i=\sum_{i=1}^N \lambda_i q_i \\
    &=&\frac{1}{\abs{K}}\sum_{i=1}^N \lambda_i \norm{a_i \wedge a_{i+1}} \\
    &=&\frac{1}{\abs{K}}\sum_{i=1}^N \lambda_i \norm{v_i} d_i    \\
    &=&\frac{1}{\abs{K}}\sum_{i=1}^N \norm{v'_i} d_i    
    \end{eqnarray*}
    where the last sum is recognized (see \cite[Section 5.1]{schneider}) to be the mixed area in the special case of polygons.  Then Alexandrov-Fenchel inequality (see \cite[Theorem 7.3.1]{schneider}) gives that $V(K,K')\geq \sqrt{\abs{K}\abs{K'}}$, which implies that $\abs{K'} \leq \abs{K}$ and concludes the proof.
\end{proof}

We are now ready to give a short proof of Theorem \ref{th:decreasing_area_BH}.

Since $\|\cdot\|_{st}^F$ depends continously on $F$, we assume the Finsler metric to be smooth and quadratically convex. The general statement follows by approximation.

Fix $\lambda>1$. We can find a symmetric $2N$-gon $K_N$ with (cyclically ordered) vertices  $a_1, \dots, a_{2N}$ such that each edge $a_ia_{i+1}$ is tangent to $B_{st}$ and $\abs{K_N}/\abs{B_{st}}<\lambda$. Its dual body is given by $K_N^\ast = \conv{\set{h_i}}$ where the $h_1, \dots , h_{2N} \in \partial B^*_{st}$ are the supporting functions associated to each edge $a_ia_{i+1}$. (They are automatically pairwise linearly independent and cyclically ordered). Let $f_1,\ldots,f_{2N}$ be the calibrating  functions associated to $h_1, \dots , h_{2N}$ and defined using Lemma \ref{lem:cal}. Remember that their differentials satisfy that $\|d_xf_i\|^\ast_x=1$ at almost every $x \in \T^2$ and that they are cyclically ordered. Set $p_i = \norm{a_i \wedge a_{i+1}}/\abs{K_N}$ for $i=1, \dots, N$ and observe that $\sum_{i=1}^N p_i=1$. Then
\begin{align*}
\frac{\area_{BH}(\T^2,\|\cdot\|_{st}^F)}{\lambda} &=  \frac{\pi}{\lambda\abs{B_{st}}} \\
&<  \frac{\pi}{\abs{K_N}} \\
    \text{(by the equality part in Proposition \ref{pr:estimate_wedge_product})} \quad &= \pi \sum_{1\leq i<j\leq N} p_i p_j \norm{h_i \wedge h_j}^\ast  \\
    &= \pi \sum_{1\leq i<j\leq N} p_i p_j \int_{\T^2} \norm{h_i \wedge h_j}^\ast dx_1 \wedge dx_2 \\
    &= \pi \sum_{1\leq i<j\leq N} p_i p_j \int_{\T^2} h_i \wedge h_j \\
    \text{(by Lemma \ref{le:integral_forms})} \quad &= \pi \sum_{1\leq i<j\leq N} p_i p_j \int_{\T^2} d_x f_i \wedge d_x f_j \\
    &= \pi \int_{\T^2} \sum_{1\leq i<j\leq N} p_i p_j \norm{d_x f_i \wedge d_x f_j}^\ast dx_1 \wedge dx_2 \\
    \text{(by Proposition \ref{pr:estimate_wedge_product})} \quad &\leq \int_{\T^2} \frac{\pi}{\abs{K_x}} dx_1 \wedge dx_2 \\
    &= \area_{BH}(\T^2,F),
\end{align*}
which concludes the proof as inequality holds for any $\lambda>1$.
\end{proof}

\subsection{A counterexample in the non-reversible case}\label{sec:cex}

We already know that systolic freedom holds for general Finsler metrics and Busemann-Hausdorff area, see section \ref{sec:sysfree}.
Still it is not clear if Theorem \ref{th:decreasing_area_BH} holds for non-reversible Finsler metrics. To conclude this section, we will explain a construction proposed to us by Ivanov to show that this is not the case. We would like to thank him for allowing us to expose his counterexample in this article.

    First consider on the 1-torus $\T^1 = \R / \Z$ endowed with the Finsler metric $\varphi: T\T^1 \rightarrow \R$ defined on the unitary vectors $\pm \partial_x$ by
    \begin{equation*}
        \varphi(x,\partial_x)=
        \begin{cases}
            \begin{array}{ll}
                1 &\text{for } x \in [0,\frac{1}{2} ) \\
                A &\text{for } x \in [\frac{1}{2},1 ) 
            \end{array}
        \end{cases}
        \text{ and } \qquad
        \varphi(x,-\partial_x)=
        \begin{cases}
            \begin{array}{ll}
                A &\text{for } x \in [0,\frac{1}{2} ) \\
                1 &\text{for } x \in [\frac{1}{2},1 ) 
            \end{array}
        \end{cases}
    \end{equation*}
    where $A>0$ is some positive constant.     
Now define the following Finsler metric on $\T^2$:
    \begin{equation*}
        F(v) = \sqrt{\varphi^2(d\pi_1 (v)) + dy^2 (v)}
    \end{equation*}
    for any $v \in T\T^2$, where $\pi_1 : \T^2 \rightarrow \T^1$ denotes the projection onto the first factor. The unit ball of such a Finsler metric has Lebesgue measure
    \begin{equation*}
        \abs{K_{(x,y)}} = \frac{\pi}{2} \left(1+\frac{1}{A}\right).
    \end{equation*}
    Hence, the Busemann-Hausdorff area of $(\T^2,F)$ is given by the following formula:
    \begin{equation*}
        \area_{BH}(\T^2, F) = \int_{\T^2} \frac{\pi}{\abs{K_{(x,y)}}} dxdy = \frac{2}{1+\tfrac{1}{A}}.
    \end{equation*}
    
    On the other hand, we first observe that $B_{st}$ is symmetric with respect to the reflections over the horizontal axis and the vertical axis respectively. Indeed, while the horizontal symmetry is straightforward, for the vertical symmetry just remark that in the universal cover the unit balls at a point $(x,y)$ and its translated $(x+1/2,y)$ are obtained one from each other by a vertical reflection, and then use the fact that the definition of the stable norm does not depend on the chosen based point in formula (\ref{def:stab}). As $\stnorm{(1,0)} = \frac{1+A}{2}$ and $\stnorm{(0,1)} = 1$, it implies that $B_{st} \subset [-\frac{2}{1+A},\frac{2}{1+A}] \times [-1,1]$. Therefore
    \begin{equation*}
        \abs{B_{st}} \leq \frac{8}{1+A}
    \end{equation*}
    and we get that
    \begin{equation*}
        \area_{BH}(\T^2,\|\cdot\|_{st}^F) = \frac{\pi}{\abs{B_{st}}} \geq \frac{\pi}{8} (1+A).
    \end{equation*}
    So for $A \to \infty$ we have that $\area_{BH}(\T^2, F) \to 2$ while $\area_{BH}(\T^2,\|\cdot\|_{st}^F) \to \infty$.

    The metric $F$ is not continuous, but we can smooth it to obtain a  counterexample to a possible generalization of Theorem \ref{th:decreasing_area_BH}.\\

\noindent {\bf Acknowledgements.} We would like to thank K. Tzanev for providing  Figure \ref{integerlines}, and D. Azagra for a valuable exchange. We are also indebted to S. Ivanov for his careful reading of the first draft of this article, where he detected that a result (Theorem 5.1 in the first draft version arXiv:2201.05010v1) and its proof were incorrect, and for allowing us to expose the counterexample of subsection \ref{sec:cex}. Finally we would like to thank the two anonymous referees for their careful reading and useful comments that help to improve the present paper.\\







\end{document}